\documentclass[10pt,oneside]{amsart}

\usepackage{amssymb, amscd, amsmath, amsthm, latexsym, hyperref,amsrefs}
\usepackage{pb-diagram, fancyhdr, graphicx, psfrag, enumerate}
\usepackage[text={6.6in,9.1in},centering,letterpaper,dvips]{geometry}
\usepackage{enumitem}\setlist[enumerate,1]{font=\upshape, itemsep=.5ex}\setlist[itemize,1]{font=\upshape, itemsep=.5ex}
\usepackage{pinlabel}
\usepackage{caption}
\usepackage{xcolor}

\def\Z{{\mathbb Z}}

\def\R{{\mathbb R}}
\def\C{{\mathbb C}}

\def\calm{\mathcal{M}}

\newcommand{\compactlist}{\begin{list}{\enumerate}{\setlength{\leftmargin}{1em}}}

\def\cs{\mathbin{\#}}
\def\co{\colon\thinspace}

\newtheorem{theorem}{Theorem}

\newtheorem{corollary}[theorem]{Corollary}

\theoremstyle{definition}

\AtBeginDocument{%
\def\MR#1{}
}

\begin{document}
\title{Signed Clasp Numbers and   Four-Genus Bounds}
\author{Charles Livingston}
\thanks{This work was supported by a grant from the National Science Foundation, NSF-DMS-1505586.   }
\address{Charles Livingston: Department of Mathematics, Indiana University, Bloomington, IN 47405}\email{livingst@indiana.edu}


\begin{abstract}   There exist knots having     positive and negative four-dimensional clasp numbers zero but having four-genus, and hence clasp number, arbitrarily large.  Such examples were first constructed by Allison Miller, answering a question of Juh\'asz--Zemke.  Further examples are constructed here, complementing those of Miller in that they are of infinite order in the concordance group, rather than being two-torsion.
\end{abstract}

\maketitle


\section{Introduction}

Let $q = p^2$,  where $p$ is an odd  prime integer. We consider the two-bridge knot $B(q,2)$, abbreviated $ B_q$, which can also be described as the $k$--twisted positive Whitehead double of the unknot, $D_+(U,k)$, where $k= (q-1)/4$.  Figure~\ref{fig-Bq} is an  illustration of $B_q$ in which the $k$ in the box denotes $k$ full right-handed twists.  If $k=1$ in the diagram, the resulting knot is the figure eight.   We prove the following theorem, which holds in  the smooth and   topological locally-flat catagories.

\begin{theorem} \label{thm-main2} If $p \ge 5$ is prime and $q = p^2$, then there exists a real number $c_p >0$ such that   the four-genus satisfies $g_4(nB_q) \ge c_pn$ for all $n >0$.
\end{theorem}

\begin{figure}[!htbp]
\labellist
\pinlabel{$k$} at 150 245
\endlabellist
\includegraphics[scale=.35]{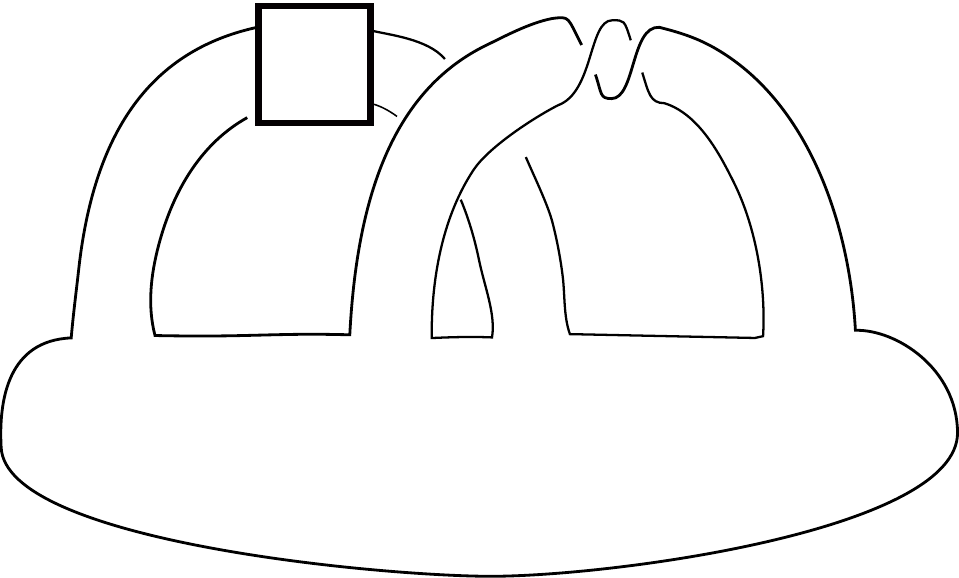}
\vskip.1in
\caption{The knot $B_q$, where $k = (q-1)/4$ denotes full right-handed twists. }
\label{fig-Bq}
\end{figure}

Finding this result was motivated by a question of Juh\'asz--Zemke~\cite{juhasz2020} concerning signed four-dimensional clasp numbers.   Let   $c(K)$ denote the  {\it four-dimensional clasp number\,}: this is the minimum value  of $m$  for which $K$ bounds  a smooth immersed disk in $B^4$ with $m$ double points. Let   $c^+(K)$ denote the minimum value  of $m$  for which $K$ bounds a smooth disk in $B^4$ with $m$ positive double points, and define  $c^-(K)$ similarly, minimizing negative double points.
In~\cite{juhasz2020} it was asked whether $c(K) - (c^+(K) + c^-(K))$ can be arbitrarily large.

Miller~\cite{MR4453695} provided the first examples answering the Juh\'asz--Zemke question positively.  It follows from Theorem~\ref{thm-main2} that the knots $B_q$ are also examples.  From the diagram it is clear that $B_q$, and hence $nB_q$, can be unknotted using only positive, or only negative, crossing changes.  Hence, $c^\pm(nB_q) = 0$.  If $nB_q$ bounds a disk in $B^4$ with $a$ double points, then those double points could be resolved to form an embedded surface of genus $a$;  it follows that $g_4(nB_q) \le c(nB_q)$ and thus Theorem~\ref{thm-main2} implies $c(nB_q) \ge  c_pn$.

The examples of this paper  are complementary to Miller's. The examples of~\cite{MR4453695}  are all amphichiral knots and thus satisfy $c^+(2K) = c^-(2K) = c(2K) = 0$; stated differently, the knots are of order two in the knot concordance group.  In contrast, the fact that $c(nB_q) >0 $ for all $n >0$ implies that $B_q$ is of infinite order in the concordance group.

\smallskip

\noindent{\it{Remarks.} }
\begin{itemize}
\item The key ingredients of the proof of Theorem~\ref{thm-main2} come from the work of Casson and Gordon~\cite{MR900252} and Gilmer~\cite{MR656619}.  The knot $B_{25}$ appears in~\cite{MR900252} as the   first example of an algebraically slice knot that is not slice: that is, $g_4(B_{25}) >0$.  A theorem of Jiang~\cite{MR620010}  demonstrates the  linear independence of the set   $\{B_{p^2}\}_{p \ge 5}$ in concordance, thus implying that $g_4(nB_{p^2}) >0$ for all $p \ge 5$ and all $n \ge 1$.  It is a theorem from ~\cite{MR656619} that permits Jiang's result to be improved to give a linearly increasing genus bound.

\item The proof of Theorem~\ref{thm-main2} provides a specific value of $c_p$ that is close to, but always less than, $1/2$. For instance, we find $c_5 = 1/4$ and $c_7 = 5/14$.  Finding any genus one knot $K$ for which $c^+(K) = c^-(K) = 0$ and $g_4(nK) \ge n/2$ for all $n \ge 1$ appears to be especially challenging.

\item The standard Seifert surface $F_q$ for $B_q$ contains  simple closed curves of framing  $\frac{q-1}{4}$ and $-1$ that are unknotted in $S^3$.  Using these curves we can construct an unknotted essential curve $\alpha$ on the Seifert surface $G$ for  $ B_q \cs \frac{ q-1}{4}B_q    = \frac{ q+3}{4}B_q$ of Seifert framing 0.  Surgery can be performed on  $G$ along $\alpha$ to produce a  surface in $B^4$ bounded by $\frac{ q+3}{4}B_q$   of genus one less than the genus of $G$; that is, it is  of genus $\frac{q-1}{4}$.  Thus, $g_4(\frac{ q+3}{4}B_q) \le \frac{q-1}{4}$.   It follows that $g_4(nB_q)$ is asymptotically bounded above by $\big(\frac{q-1}{q+3}\big)n$, 

\item  The invariants studied here are all knot concordance invariants.  From a modern perspective it would be interesting to prove the analog of Theorem~\ref{thm-main2} for a family of topologically slice knots.

\end{itemize}

\smallskip
\noindent{\it Acknowledgements.}  I appreciate helpful feedback from Pat Gilmer and Allison Miller.  Comments of a referee of an  early version of this paper led to significant improvements.


\section{Casson-Gordon invariants and four-genus bounds}\label{sec-cg}
For a knot  $K$, let $M_2(K)$ denote its 2--fold branched cover and let $\chi\co H_1(M_2(K)) \to \C^*$ be a character  taking values in the group of units generated by $e^{2 \pi i/q}$, where $q$ is a prime power.  (Such characters are naturally identified with homomorphism $\chi\co H_1(M_2(K)) \to \Z_q$.)   In~\cite{MR900252}, Casson and Gordon defined two rational-valued  invariants, $\sigma(K,\chi)$ and $\sigma_1 \tau (K,\chi)$.  The first is more readily computable in the case that $M_2(K)$ is a lens space; the second    provides an obstruction to   a knot   being slice.  They are related by the following result, an immediate consequence of~\cite[Theorem 3]{MR900252}.

\begin{theorem}  If $M_2(K)$ is a lens space and $\chi \co H_1(M_2(K)) \to \Z_q$ is a nontrivial character, then   \[ \big|  \sigma(K, \chi)  - \sigma_1 \tau(K, \chi) \big| \le 1.\]
\end{theorem}

\subsection{Computing $\sigma (B_q, \chi^r) $ for $r \ne 0  \mod p$}

We have the following result.
 
\begin{theorem}\label{thm-nonzero}
Let  $q = p^2$, where  $p$ is an odd prime.  Let $\chi$ denote a  character  that takes value  $e^{2 \pi i /p}$ on some generator of $H_1(M_2(B_q)) \cong \Z_q$.   Then
\[ \{\sigma(B_q, \chi^r)\}_{0<r< p}\ =\  \{  4r^2 - 2pr +1 \}_{0<r< p}.\]
\end{theorem}
\begin{proof}  This is essentially the key  numeric computation of~\cite{MR900252}.   The invariant  $\sigma(K,\chi)$ is defined in terms of signatures of Hermitian forms and is thus symmetic:    $\sigma(K,\chi^r) =\sigma(K,\chi^{-r}) = \sigma(K,\chi^{p-r}). $  This permits us to restrict attention to  even values of $r$:  $ \{\sigma(B_q, \chi^{2r})\}_{0<r< p/2}$.
In~\cite{MR900252}  it is shown that $\sigma(B_q, \chi^{2r} ) = 4r^2 - 2pr +1$ for $0 < r < p/2$.  (The result appears on page 196, with the values  `$`m$'' and ``$n$'' there having the value  $p$ in our application.)
\end{proof}

\subsection{Computing $\sigma_1 \tau (B_q, \chi^0) $.}

In general, there are few methods available for computing $\sigma_1\tau(K, \chi)$.  However, in the case that $K$ is of three-genus one and is algebraically slice, the invariant is determined by the Levine-Tristram signature functions of certain  knots formed as simple closed curves on a genus one Seifert surface.  This is a consequence of results related to companionship proved independently by Cooper \cite{CooperThesis},   Gilmer~\cite{MR711523}, and Litherland~\cite{MR780587}.  The paper~\cite{MR3096507} presents a more recent exposition.  We isolate the result we need.  In this statement, $\sigma_K(\omega)$ denotes the Tristram-Levine signature function defined on the unit circle in $\C^*$.

\begin{theorem}  Suppose that $K$ bounds a genus one Seifert surface $F$ and $H_1(M_2(K)) \cong \Z_q$ with $q = p^2$ for some prime $p$.   Suppose that  $\alpha$ is an essential simple closed curve on $F$ for which the $V([\alpha] ,[ \alpha]) = 0$, where $V$ is the Seifert form of $F$.   Then for $\chi \co H_1(M_2(K)) \to \Z_p \subset \C^*$, 
\[  \sigma_1 \tau (K, \chi) =    2 \sigma_{\alpha}(\zeta^r), \]
for some $r$, where $\chi(x) = \zeta \in \C^* $ for a generator $x \in H_1(M_2(K))$.
\end{theorem}

The Levine-Tristran signature function satisfies  $\sigma_K(1)  = 0$ for all $K$.  Thus we have the following corollary when applied to $\chi^0$, which  is trivial.

\begin{corollary} \label{cor-zero} Suppose that $K$ bounds a genus one Seifert surface $F$,  $H_1(M_2(K)) \cong \Z_q$, and $V(\alpha, \alpha) = 0 $ for a simple closed curve  $\alpha $ representing a nontrivial homology class, $[\alpha] \in H_1(F)$.  Then $\sigma_1 \tau (K, \chi^0 ) =   0$ for all $\chi$.
\end{corollary}

\subsection{Bounds on  $\sigma_1 \tau (B_q, \chi^r) $.}
\begin{theorem}\label{thm-boundscg} Assume $q = p^2$ where $p \ge 5$ is an odd   prime. 
\begin{itemize}

\item There exists a  generator  $\chi$ of the group of order $p$ characters on $H_1(M_2(B_q))$   such that $\sigma_1 \tau(B_q, \chi) \le   \frac{9 - p^2}{4}$.
\item  $\sigma_1 \tau (B_q,  \chi^r) \le 0 $ for all $r$.
\end{itemize}
\end{theorem}
\begin{proof}  We consider the function   $f(r)  = 4r^2 - 2pr +1$ that appears in Theorem~\ref{thm-nonzero}  as a real quadratic in the variable $r$.
Its minimum occurs at $p/4$.
The closest integer point to $p/4$ is either $(p-1)/4$ or $(p+1)/4 $ depending on whether $p \equiv 1 \mod 4$ or $  p \equiv  3 \mod4$.   In both cases the value at this  point is $(5 - p^2)/4 < -1$.  Since $\sigma(B_q, \chi^r)$ and $\sigma_1\tau(B_q, \chi^r)$ differ by at most one, we have the first statement.

For integers $r$ with  $ 1 \le r \le p/2$, the maximum value of the quadratic $f(r)$ must be at an endpoint, either $r=1$ or $r = (p-1)/2$.  We compute $f(1) = (5 - 2p)$ and $f((p-1)/2) = 2 - p$. The larger of the two is $2 -p <1$.  Even upon adding 1, this is negative.  Thus, if $\sigma_1 \tau(B_q, \chi^r)$ were to be positive for some $r$, it would have to be at $r = 0$, where the value was shown to be 0 in Corollary~\ref{cor-zero}.

\end{proof}

\section{The genus bound}

The proof of  Theorem~\ref{thm-main2}   depends on the following special case of a theorem of     Gilmer~\cite[Theorem 1]{MR656619} that relates  values of $\sigma_1\tau(K, \chi)$ to $g_4(K)$.

\begin{theorem}\label{thm:bound2}  Let $K$ be   a knot for which  $H_1(M_2(K)  ) \cong (\Z_q)^n   $, where  $q$ is a prime power.   If  $g_4(K)   \le n/2$ and the classical signature of $K$ satisfies $\sigma(K) = 0$,  then there is a subgroup $\calm \subset   (\Z_{q})^n \subset H_1(M_2(K))$ of order at least
$q^{(n- 2g_4(K))/2}$ such that for all $\chi \in \calm$,
\[ \big| \sigma_1 \tau (K, \chi) \big| \le 4g_4(K).\]
\end{theorem}\
\noindent  We will refer to the subgroup $\calm$ as a {\it metabolizer}.  In the statement of Gilmer's theorem in~\cite{MR656619} there is an additional term $\mu(K, \chi)$, but prior to the statement of that theorem he points out that $\mu(K, \chi) =  0$ in the case of characters $\chi$ of prime  power order.

\subsection{Proof of Theorem~\ref{thm-main2}}
The continuing assumption is that $q = p^2$ where $p \ge 5$ is a prime. Here is a restatement of the theorem with the value of $c_p$ specified.

\medskip

\noindent{\bf Theorem~\ref{thm-main2}.} {\it  For every odd prime $p \ge 5$,  let  $c_p =  (\frac{1}{2}  - \frac{8}{p^2+7})n$. Then  for $q = p^2$,  $g_4(nB_q) \ge c_pn$.}
\begin{proof}
We will  first assume that $n$ is such that  $g_4(nB_q) < n/2$ and find a  value of $c_p < 1/2$ for which $g_4(nB_q) \ge c_p n$ for all such $n$.  Then, in any cases that $g_4(nB_q) \ge n/2$ we will certainly also have that $g_4(nB_q) \ge c_p n$.

We abbreviate $g_4(nB_q) = g$.
We have the $H_1(nB_q)) \cong (\Z_q)^n$.  The metabolizer $\calm$ given by Theorem~\ref{thm:bound2} has order at least $p^{(n - 2g)}$.  Since each element in $\calm$ has order at most $p^2$, an independent set of generators of $\calm$ must have at least $p^{(n - 2g)/2}$ elements.   Since the value is an integer, we can take the ceiling and  let $d = \lceil \frac{n - 2g}{2 } \rceil$.

Represent a set of generators of $\calm$ as vectors in $(\Z_q)^n$.  Together these can be used to form  the rows of a matrix with at least  $d$ rows.  Row operations and column interchanges can convert this into a matrix for which the top left $d \times d$ block is an upper triangular matrix with nonzero diagonal entries and with the  further property that  rows corresponding to diagonal entries divisible by $p$ have all their entries divisible by $p$.  We can multiply each of the  elements of $\calm$ that correspond to these rows by some element in $\Z_q$ so that the first non-zero entry is $p$.  Further row operations can transform this so that the  top left $d \times d$ block  is diagonal with   all entries $p$.

The sum of the   vectors formed from the  first $d$ rows of that matrix  is an element of $\calm$ for which the first $d$ values are all $p$.  This element  can be multiplied by $k$  so that  the $d$ diagonal   entries are $kp$.  We can choose $k$ so that the character  $\chi$  on $H_1(B_q)$ that corresponds to $kp$ is the same as  $\chi$ given in Theorem~\ref{thm-boundscg} for which $\sigma_1 \tau(B_q, \chi) \le (9 - p^2)/4$.   The vector corresponds to a character $\overline{\chi}\co H_1(M_2(nB_q) ) \to \C^*$ taking values among $p$--roots of unity.

Applying the fact that  $\sigma_1 \tau(B_q, \chi)  \le 0$ for all $\chi$, along with the additivity of $\sigma_1 \tau$ (see~\cite{MR711523}), after taking absolute values  we have
\[
d \big( \frac{p^2 - 9}{4}  \big)= \frac{ (n - 2g)(p^2 -  9)}{8 }\le \big|  \sigma_1\tau(nB_q, \overline{\chi}) \big|  \le 4 g,
\]
where the second inequality comes from Theorem~\ref{thm:bound2} .

Solving for $g$ we find
\[ g \ge\big( \frac{p^2 -9}{2p^2+ 14}\big) n = (\frac{1}{2}  - \frac{8}{p^2+7})n. \]
\end{proof}


\section{Observations and questions}

\begin{enumerate}
\item  {\it The stable clasp number.}  A  function  $f \co \Z_{\ge 0} \to \R_{\ge 0}$ is called {\it subadditive} if $f(a+b) \le f(a) +f(b)$ for all $a$ and $b$.  For any such function,    $\lim_{n\to \infty}f(n)/n$ exists.  In~\cite{MR2745668} this is used to define  the stable four-genus of a knot $K$:  $g_s(K) = \lim_{n \to \infty}g_4(nK)/n$.  In the exact same way, one can define the stable clasp number of a knot $K$ to be $c_s(K) = \lim_{n \to \infty}c(nK)/n$.  For Miller's examples~\cite{MR4453695}, $c_s(K) = 0$.  We have  
\[  \frac{q-9}{2q - 14} \le c_s(B_q) \le 1 .
\]

{\bf Problems.}   Determine $c_s(B_q)$ exactly. Find any knot $K$ for which  $c_s(K)$  for which $c_s(K) \not \in \Z$.

\item   Find topologically slice knots $K_n$ for which $c(K_n) - (c^+(K_n)+ c^-(K_n))$ goes to infinity as $n$ increases.  Can such example be found for which $c^+(K_n) = 0= c^-(K_n)$ for all $n$?

\item The examples in this paper and those in~\cite{MR4453695} depended on estimates of the four-genus.  Are there examples of knots $K$ for which $c^+(K) = 0 = c^-(K)$ and $c(K) > g_4(K)$?
\end{enumerate}



\bibliography{../../../../BibTexComplete}
\bibliographystyle{plain}	

\end{document}